\documentclass[11pt]{article}
\usepackage{amssymb,amsmath,setspace,graphicx,multicol,wrapfig,amsfonts,amsthm,titling,url,array,parskip,enumerate,amsthm}
\usepackage{hyperref}
\usepackage{placeins}
\usepackage{aurical,pbsi}
\usepackage[T1]{fontenc}
\usepackage[hmargin=2.5cm,vmargin=2.5cm]{geometry}

\numberwithin{equation}{section}

\makeatletter
\def\thm@space@setup{%
  \thm@preskip=0.1in
  \thm@postskip=0in
}
\makeatother

\theoremstyle{plain}
\newtheorem{thm}{Theorem}[section]

\newtheorem{prop}[thm]{Proposition}
\newtheorem{cor}[thm]{Corollary}

\theoremstyle{definition}
\newtheorem{defn}{Definition}[section]

\newtheorem*{rem}{Remark}

\DeclareMathOperator{\wt}{wt}
\DeclareMathOperator{\shape}{shape}

\begin{document}

\begin{center} {\Large{\sc A Determinantal Formula for Catalan Tableaux and TASEP Probabilities}} \\
\vspace{0.1in}
Olya Mandelshtam
\footnote{Phone: +1 (949) 689-5748\\
Fax: +1 (510) 642-8204\\
E-mail: olya@math.berkeley.edu}
 \end{center}

\vspace{0.3in}
\begin{abstract}
We present a determinantal formula for the steady state probability of each state of the TASEP (Totally Asymmetric Simple Exclusion Process) with open boundaries, a 1D particle model that has been studied extensively and displays rich combinatorial structure. These steady state probabilities are computed by the enumeration of Catalan tableaux, which are certain Young diagrams filled with $\alpha$'s and $\beta$'s that satisfy some conditions on the rows and columns.  We construct a bijection from the Catalan tableaux to weighted lattice paths on a Young diagram, and from this we enumerate the paths with a determinantal formula, building upon a formula of Narayana that counts unweighted lattice paths on a Young diagram. Finally, we provide a formula for the enumeration of Catalan tableaux that satisfy a given condition on the rows, which corresponds to the steady state probability that in the TASEP on a lattice with $n$ sites, precisely $k$ of the sites are occupied by particles. This formula is an $\alpha\ /\  \beta$ generalization of the Narayana numbers.
\end{abstract}

\section{Introduction}

The TASEP (Totally Asymmetric Simple Exclusion Process) is a model from statistical mechanics that describes particles hopping in one direction along a one-dimensional lattice. New particles can enter and exit the sides of the lattice, and particles can hop to the right as long as there is at most one particle per site. 

\begin{wrapfigure}[7]{r}{0.35\textwidth}
\centering
\includegraphics[width=0.35\textwidth]{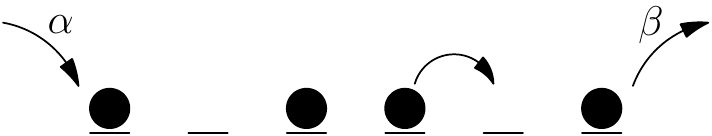}
\caption{The parameters of the TASEP.}
\noindent
\label{TASEP_parameters}
\end{wrapfigure}

In the TASEP with open boundaries, with parameters shown in Figure \ref{TASEP_parameters}, the rules are as follows: particles hop right on a lattice of $n$ sites, such that 
\begin{itemize}
\item There is at most one particle per site.
\item A new particle can enter at the left at rate $\alpha$.
\item A particle can exit at the right at rate $\beta$.
\item Particles hop to the right at rate 1.
\end{itemize}

We represent a state of a TASEP of size $n$ by a word $\tau$ in $\{\bullet,\circ\}^n$ where $\bullet$ represents a particle and $\circ$ represents a hole. We use the notation $\Pr(\tau)$ to denote the stationary probability of state $\tau$.

The TASEP is a special case of the ASEP (Asymmetric Simple Exclusion Process), in which particles can hop both right and left, and can enter and exit from both sides of the lattice. The ASEP is one of the simplest and most investigated models for the dynamics of particle systems. The existence of exact solutions for this system is extremely useful as testing grounds for non-equilibrium problems in statistical mechanics (see Derrida \cite{derrida} and the references therein). The ASEP also displays rich algebraic and combinatorial structure: in particular, there are many nice combinatorial results for steady state probabilities in various levels of generality, due to Shapiro and Zeilberger \cite{shapiro}, Duchi and Schaeffer \cite{duchi}, and Corteel and Williams in \cite{williams2007} and \cite{williams2011}. From these works, there is a large number of related objects that give a combinatorial interpretation for the steady state probabilities of some specializations of the ASEP. However, in general there is not an explicit formula for the steady state probabilities. 

In this work we limit the discussion to the case of the TASEP, and we provide an explicit determinantal formula for all steady state probabilities. We start by defining \textbf{Catalan tableaux}, which can be seen to be equivalent to the staircase tableaux of Corteel and Williams.

\begin{defn} \label{catalan_def}
A \textbf{Catalan tableau} of \textbf{size} $n$ is a filling of the Young diagram of shape $(n,n-1,\ldots,1)$ with $\alpha$'s and $\beta$'s  according to the following rules:
\begin{enumerate}[i.]
\item Every box on the diagonal must contain an $\alpha$ or a $\beta$.
\item All the boxes in the same row and west of a $\beta$ must be empty.
\item All the boxes in the same column and north of an $\alpha$ must be empty.
\item Any box that does not have either an $\alpha$ in its column below or a $\beta$ in its row to the right must contain an $\alpha$ or a $\beta$.
\end{enumerate}
\end{defn}

\begin{defn} The \textbf{type} of a Catalan tableau is the word in $\{\bullet,\circ\}^{\ast}$ that is obtained by reading the diagonal entries from top to bottom, where $\alpha$ is read as a $\bullet$ and $\beta$ is read as a $\circ$.  The type of a Catalan tableau can be interpreted as a state of the TASEP.
\end{defn}

\begin{defn} The \textbf{weight} $\wt(T)$ of a Catalan tableau $T$ is the product of the symbols in its filling. See Figure \ref{staircase} for an example.
\end{defn}
From the work of Corteel and Williams, we have the following beautiful interpretation for the TASEP in terms of the Catalan tableaux.
\begin{thm}[Corteel, Williams] \label{cw_thm}
Let $\Pr(\tau)$ be the stationary probability of state $\tau$ of a TASEP of size $n$. Then
\begin{displaymath}
\Pr(\tau) = \frac{1}{Z_n}\sum_T \wt(T)
\end{displaymath}
where the sum is over Catalan tableaux $T$ of type $\tau$, and $Z_n = \sum_T \wt(T)$ is the sum over all Catalan tableaux $T$ of size $n$.
\end{thm}

\begin{wrapfigure}[14]{r}{0.31\textwidth}
\centering
\includegraphics[width=0.27\textwidth]{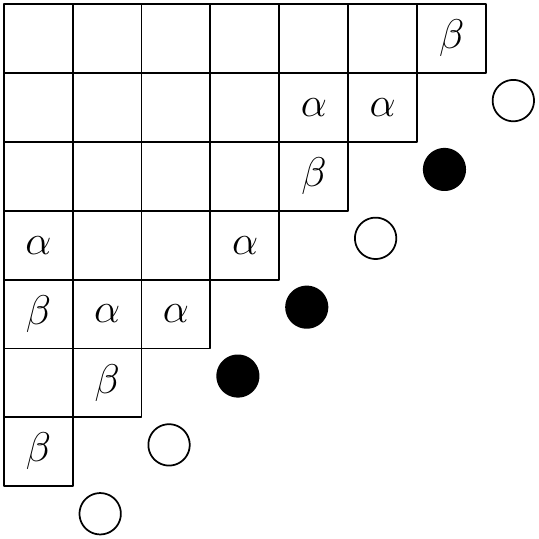}
\caption{A Catalan tableau of type $\circ\bullet\circ\bullet\bullet\circ\circ$ and weight $\alpha^6\beta^5$.}
\noindent
\label{staircase}
\end{wrapfigure}

Our main results are new formulas for the enumeration of Catalan tableaux, which in turn provide new formulas for steady state probabilities of the TASEP.  For a given word $\tau$ in $\{\bullet,\circ\}^{\ast}$, define the weight generating function
\begin{equation*}P_{\tau}(\alpha,\beta) = \sum_T \wt(T),\end{equation*} 
where the sum is over Catalan tableaux $T$ of type $\tau$.
Our first result is a determinantal formula for $P_{\tau}(\alpha, \beta)$, which is given in Theorem \ref{main_thm}.  Our method of proving this result is to give a bijection between Catalan tableaux and certain weighted lattice paths which resemble a construction of Viennot \cite{viennot}, and then enumerate the weighted lattice paths, generalizing an argument of Narayana.  As a corollary, we obtain a determinantal formula for the steady state probability of being in an arbitrary state of the TASEP. Our second result, Theorem \ref{mbyk_thm}, is an explicit expression for the Catalan tableaux whose diagonal contains a fixed number of $\alpha$'s and $\beta$'s.  This formula is in fact a 2-parameter generalization of the Narayana numbers (which are related to the Catalan numbers).  As a corollary, we obtain an explicit formula for the steady state probability that in the TASEP on a lattice of $m+k$ sites, precisely $k$ of the states are occupied by particles.

In Section \ref{sec_catalan} of this paper, we describe some properties of the Catalan tableaux and provide a more convenient characterization of them in terms of a compact version that we call Condensed Catalan tableaux. In Section \ref{sec_bijection} we define the bijection from Catalan tableaux to weighted paths which is central to our main results. In Section \ref{sec_narayana}  we provide a proof of Narayana's determinantal formula, and then give an analogous proof for enumerating the weighted paths that represent the Catalan tableaux with the $\alpha,\beta$ generalization $P_{\tau}(\alpha,\beta)$ in Section \ref{sec_determinantal}. Finally, Section \ref{sec_enum} contains a formula for the number of Catalan tableaux with a given number of $\alpha$'s and $\beta$'s on the diagonal, and the related corollaries. 

\medskip
\noindent {\bf Acknowledgements.} I would like to thank Lauren Williams for suggesting the problem to me, and for numerous helpful conversations. I would also like to thank the anonymous referees who gave some very detailed and useful comments. The author was supported by the NSF grant DMS-0943745. DMS-1049513

\FloatBarrier
\section{Condensed Catalan tableaux} \label{sec_catalan}

First, we give some intuition for the structure of Catalan tableaux. An immediate way to increase the size of a Catalan tableau is to add a new column to its left (or a new row above). We define a \textbf{free row} of a Catalan tableau to be a row that is indexed by $\alpha$. This means that the leftmost symbol in this row is an $\alpha$, and hence it contains no $\beta$'s. Analogously, we define a \textbf{free column} to be a column that is indexed by $\beta$, which means the top-most symbol in the column is a $\beta$, and so it contains no $\alpha$'s. To increase the size of the tableau by adding a new column to the left, symbols $\alpha$ or $\beta$ in the new column can only be in the locations of the free rows. Due to (iv) of definition \ref{catalan_def}, the only allowed empty boxes are precisely those that lie above an $\alpha$, left of a $\beta$, or both. Hence every new column that we add must be, starting from the bottom, a (possibly empty) sequence of $\beta$'s followed by an $\alpha$, or just a sequence of $\beta$'s, such that every free row is occupied by a $\beta$ until the $\alpha$ is reached. Figure \ref{cat_example} shows the three cases for the allowed additions of a new column to a Catalan tableau.

\begin{figure}[h]
\centering
\includegraphics[width=0.8\textwidth]{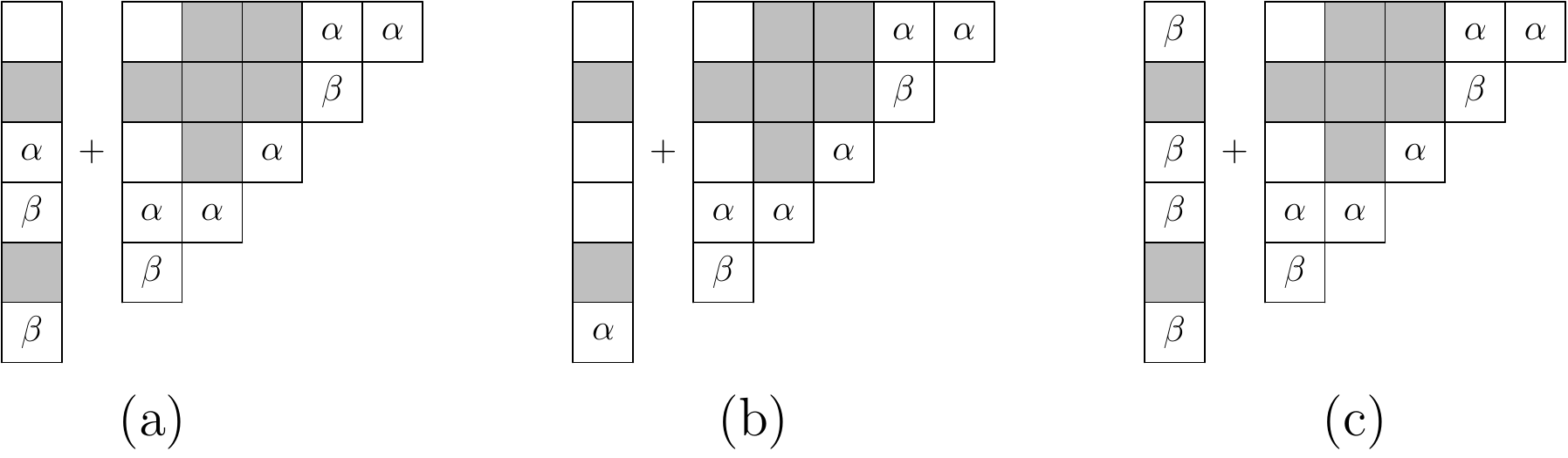}
\caption{The new column we add can have (a) a sequence of $\beta$'s followed by an $\alpha$, (b) a single $\alpha$, or (c) a $\beta$ in every free row.}
\noindent
\label{cat_example}
\end{figure}


Before we can define the lattice paths which we call \textbf{Catalan paths} that are central to our main results, we must first introduce another characterization of the Catalan tableaux in terms of a condensed version. 

\subsection{Condensed Catalan tableaux}

A partition $\lambda=(\lambda_1,\lambda_2,\ldots)$ with $\lambda_1 \geq \lambda_2 \geq \cdots \geq 0$ is a weakly decreasing sequence of nonnegative integers. We identify a partition with its Young diagram, which is a collection of left-adjusted rows of boxes such that the $i$th row contains $\lambda_i$ boxes.

\begin{defn}\label{condensed_defn}
A \textbf{Condensed Catalan tableau} $T$ of \textbf{size} $(k,n)$ and \textbf{shape} $\lambda$ is a filling of the Young diagram $\lambda$ with $\alpha$'s and $\beta$'s according to rules (ii)--(iv) of Definition \ref{catalan_def}. The Young diagram $\lambda$ is contained in a $k \times (n-k)$ rectangle, justified to the northwest.
\end{defn}

We use the notation $\shape(T)$ to mean the shape of the Young diagram $\lambda$ assotiated to $T$. We also associate to $T$ a lattice path $L=L(T)$ with steps south and west, which starts at the northeast corner of the rectangle and ends at the southwest corner, and follows the southeast border of $\lambda$. The \textbf{type} of $T$ is the word $\tau$ in $\{\bullet,\circ\}^{\ast}$ that we obtain by reading $L$ from northeast to southwest and assigning a $\bullet$ to a south-step and a $\circ$ to a west-step.

\begin{defn}
The \textbf{weight} of a Condensed Catalan tableau $T$ of size $(k,n)$ with associated Young diagram $\lambda$ is
\begin{equation*}\wt(T) = \alpha^{k+j} \beta^{n-k+\ell}\end{equation*}
where $j$ is the number of $\alpha$'s and $\ell$ is the number of $\beta$'s in the filling of $\lambda$. 
\end{defn}

\begin{figure}[h]
\centering
\includegraphics[width=0.6\textwidth]{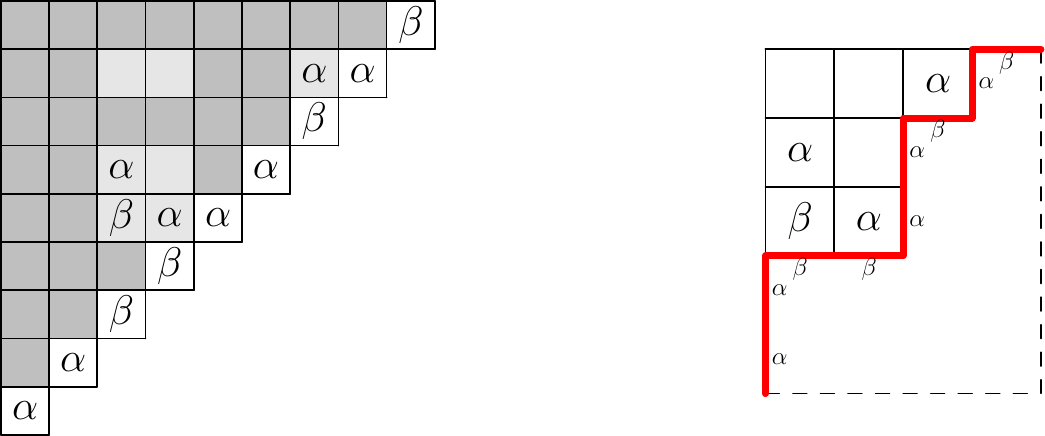}
\caption{A Catalan tableau of type $\circ\bullet\circ\bullet\bullet\circ\circ\bullet\bullet$ and its corresponding Condensed Catalan tableau. The dark boxes in the staircase version are the boxes of the rows or columns that we remove in order to form the condensed tableau, and the white boxes of the diagonal are the ones that determine the shape of the condensed tableau. The Condensed Catalan tableau has size $(k,n)=(4,9)$, shape $\shape(T) = (3,2,2,0,0)$, and weight $\wt(T)=\alpha^8\beta^5$. The path outlined in bold on the Condensed Catalan tableau is the lattice path $L(T)$.}
\noindent
\label{condensed}
\end{figure}

In Figure \ref{condensed}, we demonstrate by example the conversion from a staircase-shape Catalan tableau of size $n$ with $k$ $\alpha$'s on the diagonal to a Condensed Catalan tableau of size $(k,n)$. We remove from the staircase version of the tableau all the rows whose right-most box contains a $\beta$ and all the columns whose bottom-most box contains an $\alpha$.  After collapsing together the remaining boxes and justifying them to the northwest in a rectangle of size $k \times (n-k)$, we obtain a Young diagram $\lambda$ that is filled with $\alpha$'s and $\beta$'s according to rules (ii)--(iv) of Definition \ref{catalan_def}. The southeast border of $\lambda$ is identified with a lattice path $L$ with south and west edges, that goes from the northeast to the southeast corners of the rectangle. The edges of $L$ are given weights, with a west edge assigned a $\beta$ and a south edge assigned an $\alpha$. It is easy to check that a west edge of $L$ corresponds to a row of the staircase tableau whose right-most box contained a $\beta$, and a south edge of $L$ corresponds to a column of the staircase tableau whose bottom-most box contained an $\alpha$. Consequently, the weight of $L$ is precisely the weight of the diagonal of the staircase version of the Catalan tableau.

The following connects the Condensed Catalan tableaux back to the TASEP.

Let $\tau$ be a word of length $n$ in $\{\bullet,\circ\}$ representing a state of the TASEP. We draw a lattice path $L$ with steps south and west by reading $\tau$ from left to right, and by drawing a step south for a $\bullet$ and a step west for a $\circ$. We obtain a Young diagram $\lambda$ whose southeast border coincides with $L$. The size of the rectangle containing $\lambda$ is $k \times (n-k)$, where $k$ is the number of $\bullet$'s in $\tau$. More precisely, $\lambda=(\lambda_1,\ldots,\lambda_k)$, where $\lambda_i$ the number of $\circ$'s to the right of the $i$th $\bullet$. Then any filling with $\alpha$'s and $\beta$'s of $\lambda$ according to rules (ii)--(iv) of definition \ref{catalan_def} yields a Condensed Catalan tableau of type $\tau$. We can also refer to $\lambda$ by $\lambda(\tau)$.

\begin{rem}
Note that when $j_1$ of the $\lambda_i$'s of the Condensed Catalan tableau $T$ of type $\tau$ are equal to 0, this means that $\tau$ ends with a a string of  $j_1$ $\bullet$'s. Furthermore, if $ (n-k)-\lambda_1 = j_2$, this means that $\tau$ begins with a string of $j_2$ $\circ$'s.  Thus keeping track of the size of the rectangle containing the Young diagram associated to $T$ is important for preserving the weight of the Condensed Catalan tableau. We can see an example of this in Figure \ref{condensed}.
\end{rem}

\begin{rem}
Condensed Catalan tableaux are essentially the alternative tableaux studied by Viennot in \cite{slides}. See also \cite{viennot} for a closely related object.  Viennot \cite{viennot} states a further characterization of the steady state probabilities that is given by the enumeration of certain weighted lattice paths, which we call Catalan paths and define in the following section. A specialization of this result for the case $\alpha=\beta=1$ is presented in \cite{shapiro}.
\end{rem}

From this point on, we will identify the staircase version of the Catalan tableaux with their corresponding Condensed Catalan tableaux since they are equivalent with a simple bijection.

\section{Bijection from Catalan tableaux to weighted lattice paths}  \label{sec_bijection}
\begin{defn} Let $L=L(T)$ be the lattice path contained in a $k \times  (n-k)$ rectangle that represents a Catalan tableau $T$, and let $\lambda=\shape(T)$ be the Young diagram whose southeast border coincides with $L$. A \textbf{Catalan path} that is \emph{constrained by} $L$ is a path that starts from the northeast end of $L$ and ends at the southwest end, taking the steps south and west, and never crossing $L$. To every such path, we associate a unique labeling of its steps by $\alpha$, $\beta$, and 1 as follows:
\begin{itemize}
\item A south-step that does not lie on the west boundary of $\lambda$ receives a $\beta$.
\item A south-step that lies on the west boundary of $\lambda$ receives a 1.
\item A west-step that lies strictly above $L$ receives an $\alpha$.
\item A west-step that coincides with $L$ receives a 1.
\end{itemize}
Such a path is called a \textbf{weighted Catalan path}, and its weight is the product of all the weights of its edges. In Figures \ref{ex1} and \ref{path_bijection} (c) we see examples of weighted Catalan paths.

\end{defn}

\begin{rem}
The weight of the Catalan path associated to a Condensed Catalan tableau is equal to the weight of the \emph{filling} of that Condensed Catalan tableau. So, to get the total weight of the tableau, we take the product of the weight of the Catalan path times the weight of the boundary path $L$. Recall that the weight of $L$ is $\alpha^k \beta^{n-k}$ where $k \times  (n-k)$ is the size of the rectangle containing $L$.
\end{rem}

\begin{figure}[h]
\centering
\includegraphics[width=0.3\textwidth]{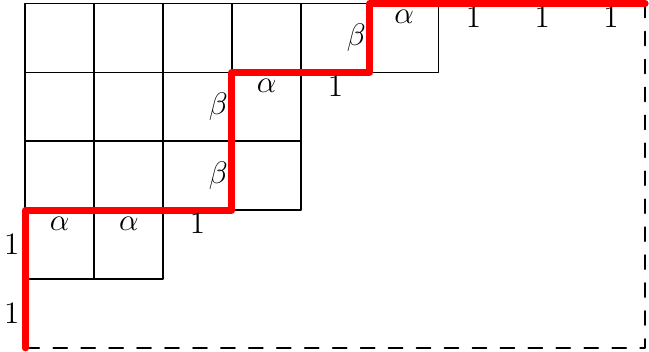}
\caption{An example of a Catalan path with weight $\alpha^4\beta^3$ in a Young diagram of shape $(6,4,4,2,0)$, contained in a $5 \times 9$ rectangle. The total weight of the tableau with the path is thus $\alpha^9\beta^{12}$.}
\noindent
\label{ex1}
\end{figure}

After fixing the size $(k,n)$ of a Condensed Catalan tableau $T$, we can identify the path $L(T)$ with the shape $\lambda=\shape(T)$. Thus we can say a Catalan path is constrained by $\lambda$ to mean the path is constrained by $L$.

\begin{prop}\label{bijprop} There is a weight-preserving bijection between Condensed Catalan tableaux of a fixed shape $\lambda$ and weighted Catalan paths constrained by the same shape $\lambda$. 
\end{prop}

For the purpose of the proof, it will be useful to make the following definition.

\begin{defn} A \textbf{modified Catalan tableau} is a filling of a Young diagram $\lambda$ with $\alpha$'s and $\beta$'s with the following properties:
\begin{enumerate}[i.]
\item \label{a} There is at most one $\beta$ in each row.
\item \label{d} The $\beta$'s must all be in consecutive rows starting from the top one. (If there is a $\beta$ in some row, there must also be a $\beta$ in the row above it.)
\item \label{b} There cannot be a $\beta$ to southeast of another $\beta$.
\item \label{c} There cannot be an $\alpha$ to the left and in the same row as a $\beta$.
\item Every box adjacent to the southeast boundary of $\lambda$ contains an $\alpha$ if permissible according to Property \ref{c}.
\end{enumerate}
\end{defn}

Notice that Properties (\ref{a})--(\ref{b}) imply that if the $\beta$'s are associated with south-steps on the edge directly to the east of those $\beta$'s, they will form a Catalan path.

\begin{proof}[Proof of Proposition~\ref{bijprop}]
We obtain the bijection from Condensed Catalan tableaux to Catalan paths by first constructing a modified Catalan tableau that is in bijection with both the Condensed Catalan tableaux and the Catalan paths. The direct correspondence between modified Catalan tableaux and Condensed Catalan tableaux of the same shape can be observed in the example given in Figure \ref{path_bijection} (b). The modified Catalan tableau is constructed as follows from a given Condensed Catalan tableau.
\begin{itemize}
\item Drop each $\alpha$ of the Condensed Catalan tableau to the bottom of the column that contains it. 
\item Reading the $\beta$'s from right to left, place them at the highest row possible (within the same column) such that there is at most one $\beta$ per row, so that in the end we have some set of $\alpha$'s lining the lower boundary of the tableau, and a set of $\beta$'s such that if read right to left, they will be decreasing in height.
\end{itemize}

\begin{figure}[h]
\centering
\includegraphics[width=2.5in]{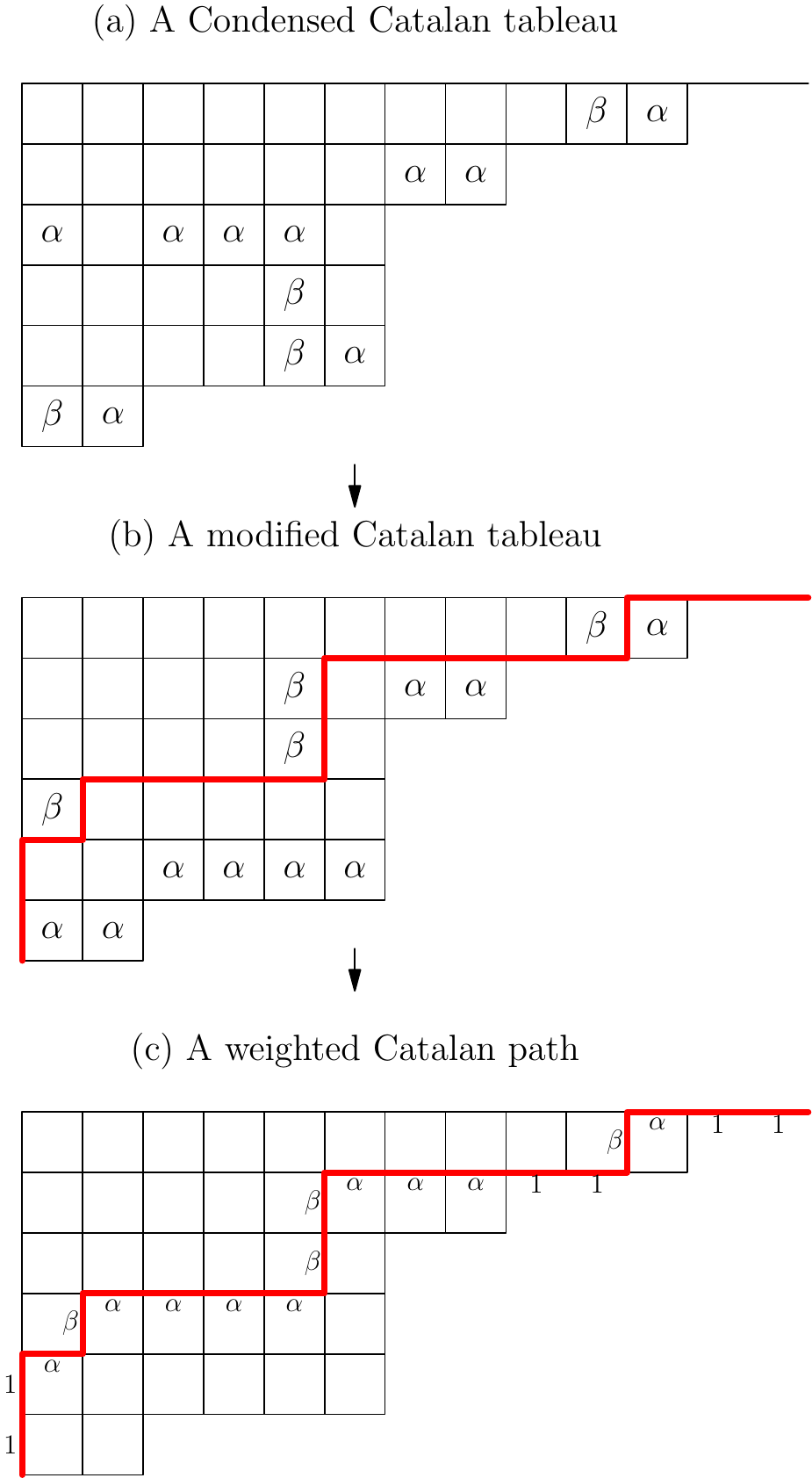}
\caption{A Catalan filling of shape $(11,8,6,6,6,2)$ in a $6 \times 13$ rectangle that corresponds to a weighted Catalan path  of weight $\alpha^9\beta^4$ on this shape.}
\noindent
\label{path_bijection}
\end{figure}

For each column, the number of boxes below the lowest $\beta$ is the number of free rows (i.e. $\alpha$-indexed rows) remaining after the preceding column. Hence the construction of the modified Catalan tableau keeps a record of the number of free rows in each column of the Condensed Catalan tableau. From Figure \ref{cat_example}, we see that the number of free rows remaining after each column of the Condensed Catalan tableau uniquely determines its filling. 
 
A weighted Catalan path can now be constructed from the modified Catalan tableau by reading the columns right to left:
\begin{itemize}
\item If the column contains some $\beta$'s, we draw a south-step that is labelled by $\beta$ on the edge directly to the right of each $\beta$. This labeling is Catalan-path consistent, since these south-steps are all to the right of the west boundary of the Young diagram by construction. 
\item If that column also contains an $\alpha$, we continue the path with a west step that is labelled by $\alpha$. Note that since the $\beta$'s (and the corresponding south-steps) in the modified Catalan tableau are strictly above the southeast boundary as they all lie above an $\alpha$ in that column, this new west step will be strictly above the southeast boundary of the Young diagram. Thus the labeling of the new step by an $\alpha$ is Catalan-path consistent. 
\item If that column contains no $\alpha$, we continue the path with a west step labeled by 1. Notice that if that column contains no $\alpha$, then there are no free rows remaining after it. Thus the height of the lowest $\beta$ in that column is zero. Consequently, this west step must lie on the southeast boundary of the Young tableau, and so its labeling with a 1 is Catalan-path consistent.
\item Once all the columns have been read, the path is at the west boundary of the Young diagram, so we complete the path by drawing some down-edges labeled by 1 directly down to the southwest corner of the tableau. This segment is by construction Catalan-path consistent.
\end{itemize}

In the other direction, given a weighted Catalan path, for each vertical segment that is not on the west boundary of the Young diagram, draw a $\beta$ in the box directly to the left. Then, reading from right to left, for each column, drop the $\beta$'s to the lowest possible locations in that same column (i.e. to the bottom-most free row) so that in the end there is at most one $\beta$ per row. Then, fill in $\alpha$'s in the first available spots reading from right to left, so that in the end there is at most one $\alpha$ per column. Figure \ref{path_bijection} shows an example of a Catalan tableau and its corresponding weighted Catalan path.

\end{proof}

\begin{rem} Viennot gives a similar weight-preserving bijection between Catalan tableaux and lattice paths in Equation (5.3) of \cite{viennot}. However, he does not explicitly assign weights to the edges of the lattice path.
\end{rem}

\FloatBarrier
\section{Narayana's path-counting formula} \label{sec_narayana}

Narayana \cite{narayana} provided the following formula for counting the number of Catalan paths on a Young diagram $\lambda$. An example of such a path on a tableau of shape $(6,4,4,2)$ is shown in Figure \ref{ex1}.

\begin{thm}[Narayana]
The number of Catalan paths on a Young diagram of shape $\lambda=(\lambda_1,\ldots,\lambda_k)$ is $\det A_{\lambda}$, where
\[
A_{\lambda}= \bigg( \binom{\lambda_j+1}{j-i+1} \bigg)_{1\leq i,j\leq k}
\]
\end{thm}

From \cite{viennot}, Narayana's determinantal formula enumerates the unweighted Catalan paths constrained by the Young diagram of shape $\lambda$.

\begin{wrapfigure}[12]{r}{0.3\textwidth}
\centering
\includegraphics[width=0.3\textwidth]{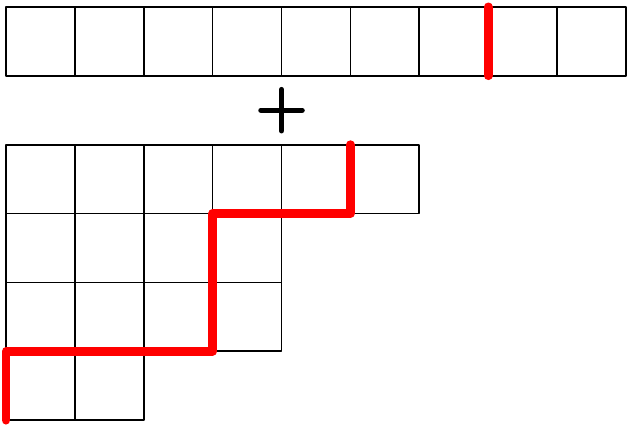}
\caption{Building a path by adding a row above an existing path.}
\noindent
\label{path1}
\end{wrapfigure}

We include the proof of Narayana's Theorem as a warmup for the proof of our main result, Theorem \ref{main_thm}. 

\begin{proof}
We prove Narayana's formula using induction on the number of rows.

First, if the Young diagram has a single row of length $\lambda_1$, all paths in that diagram have exactly one south-step, for which there are $\lambda_1+1$ possible locations including the left-most and right-most edges of the diagram. Thus there are $\lambda_1+1$ such paths, which equals the single entry which is also the determinant of the matrix $A_{(\lambda_1)}$.

Now, let us assume the formula above holds for counting the paths in all diagrams containing up to $k-1$ rows. We first observe that the bottom right $m\times m$ minor of $A_{(\lambda_1,\ldots,\lambda_k)}$ equals the matrix $A_{(\lambda_{k-m+1},\lambda_{k-m+2},\ldots,\lambda_k)}$, and so its determinant counts the number of paths on the shape $(\lambda_{k-m+1},\lambda_{k-m+2},\ldots,\lambda_k)$.

We also note that any path can be uniquely represented by its set of south-steps: to reconstruct the path from the set of south-steps, we simply connect them with west-steps. \textbf{Rule 1} below gives us an if-and-only-if condition to check whether some set of south-steps that lies within the $(\lambda_1,\ldots,\lambda_k)$ shape gives rise to a valid path. 

\textbf{Rule 1.} In a valid path, a south-step in row $\lambda_j$ cannot be west of a south-step in row $\lambda_{j+1}$. 

A path is valid if and only if there is a unique south-step in each row and Rule 1 holds for each $1 \leq j\leq k-1$.

We expand the determinant by its top row as follows:
\begin{equation*} \det A_{(\lambda_1,\ldots,\lambda_k)} =  {\lambda_1+1 \choose 1}\det A_{(\lambda_2,\ldots,\lambda_k)} -  {\lambda_2+1 \choose 2}\det A_{(\lambda_3,\ldots,\lambda_k)} +\cdots \pm {\lambda_k+1 \choose k} \det A_{\emptyset},
\end{equation*}

where by convention, $\det A_{\emptyset}=1$.

\textbf{Step 1:} A path on $(\lambda_1,\ldots,\lambda_k)$ contains a south-step somewhere in the top row, and there are $\lambda_1+1$ choices for that south-step, as we see in Figure \ref{path1}. In particular, $ {\lambda_1+1\choose 1}\det A_{(\lambda_2,\ldots,\lambda_k)}$ counts the combination of all choices for the south-step in the top row with all possibilities for paths that start in the upper-right corner of the shape $(\lambda_2,\ldots,\lambda_k)$ (i.e. the paths in all rows below the first one). All such combinations will certainly be counting the collections of south-steps that represent all the possible paths, but they will also be counting some illegal collections of south-steps that violate Rule 1 at rows 1 and 2, such as in Figure \ref{illegal} (a).

\textbf{Step 2:} Specifically, Rule 1 is violated at rows 1 and 2 when the south-step in the top row lies to the {\it left} of the first south-step of the path starting in the second row, such as in Figure \ref{illegal} (b). Let us subtract out those combinations. In particular, all such illegal combinations will be counted by the set of an illegal pair of south-steps in the top two rows, combined with all possible paths starting from the third row. The ways of selecting this illegal pair of south-steps in rows 1 and 2 is simply a choice of two disjoint columns $C_1$ and $C_2$ with $C_1<C_2$ such that the top row gets a south-step in column $C_1$ and the row below it gets a south-step in  column $C_2$. The number of such choices is ${\lambda_2+1 \choose 2}$, and the number of all paths starting at the third row is $\det A_{(\lambda_3,\ldots,\lambda_k)}$. Their product is the second term in the expansion of the determinantal formula.

\textbf{Step 3:} Now we have subtracted all collections of south-steps that violate Rule 1 in rows 1 and 2, but some of the terms that we subtracted were not actually counted in Step 1, so we have to add those back in. Those paths are the ones where Rule 1 is violated not only in rows 1 and 2, but also in rows 2 and 3. This is because Step 1 presumes that the collection of south-steps starting from row 2 is legal, in particular that the pair of south-steps in rows 2 and 3 is legal. Thus, we must add back in all combinations of the form shown in Figure \ref{illegal} (c). The possibilities for the top 3 rows are counted by ${\lambda_3+1 \choose 3}$, and the possibilities for all paths starting from row 4 are given by $\det A_{(\lambda_4,\ldots,\lambda_k)}$, so their product is the third term of the determinantal expansion.

\begin{figure}[h]
\centering
\includegraphics[width=0.95\textwidth]{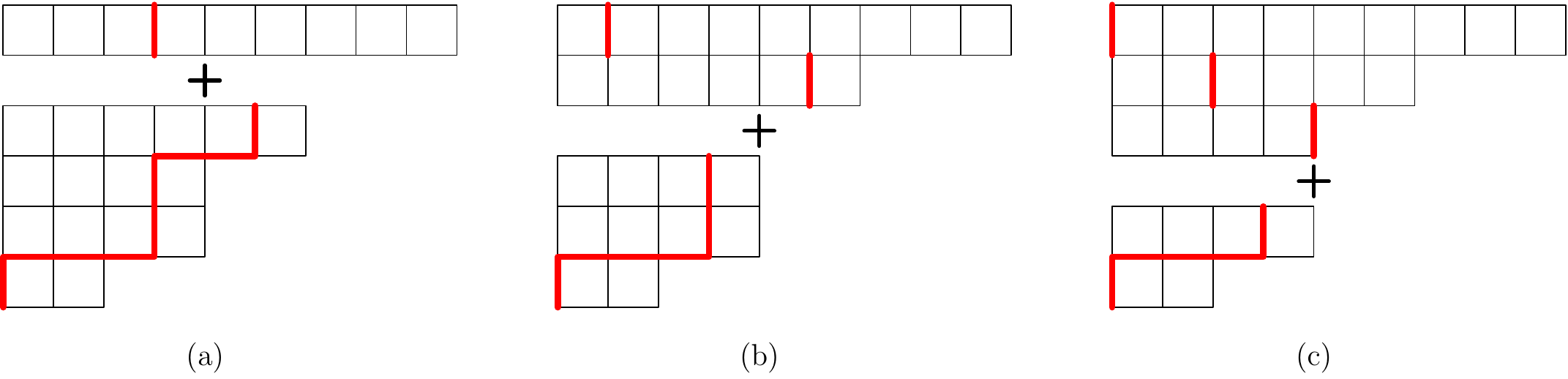}
\caption{(a) A construction in Step 1 that does not yield a valid path. (b) The collections of south-steps we subtract in Step 2. (c) The collections of south-steps we add back in in Step 3. }
\noindent
\label{illegal}
\end{figure}

\textbf{Steps 4 through $\boldsymbol{k}$:} This addition of terms is repeated for $k$ steps, where at step $j$ we add $(-1)^{j-1}$ times the terms that violated Rule 1 in rows 1 and 2, 2 and 3, $\ldots$, and $(j-1)$ and $j$. Similarly to the above, this is due to the fact that the steps 1 through $(j-1)$ only accounted for terms that had a valid collection of south-steps starting from row $(j-1)$. Thus at step $j$, we add the product $(-1)^{j-1}{\lambda_j+1 \choose j} \det A_{(\lambda_{j+1},\ldots,\lambda_k)}$, the $j$'th term in the determinantal expansion 
\begin{equation*}\det A_{(\lambda_1,\ldots,\lambda_k)} =  \sum_{j=1}^k (-1)^{j-1}{\lambda_j+1 \choose j}\det A_{(\lambda_{j+1},\ldots,\lambda_k)}.
\end{equation*}
Thus we have accounted for all the terms in Narayana's formula.
\end{proof}

\FloatBarrier
\section{Enumeration of the weighted paths corresponding to Catalan fillings of shape $\lambda$}\label{sec_determinantal}

From the bijection from Catalan tableaux to paths, we extend Narayana's determinantal formula that counts the unweighted paths to one that gives the weight generating function for the weighted Catalan paths. In this way, we construct an explicit formula for enumerating Catalan tableaux.

\begin{thm} \label{main_thm} Let $\lambda=(\lambda_1,\ldots,\lambda_k)$ with $0 \geq \lambda_k \geq \cdots \geq \lambda_1$ be the shape of a Young diagram that is contained in a $k \times (n-k)$ rectangle. The weight generating function for Condensed Catalan tableaux of size $(k,n)$ and shape $\lambda$ is 
 \begin{equation*}P_{\lambda}(\alpha,\beta) = \alpha^k\beta^{n-k} \det A_{\lambda}^{\alpha,\beta},\end{equation*} 
 where $A_{\lambda}^{\alpha,\beta}=(A_{ij})_{1\leq i,j\leq k}$ is given by
\begin{multline*}
A_{ij} =  \beta^{j-i} \alpha^{\lambda_i-\lambda_{j+1}} \left({\lambda_{j+1} \choose j-i}+\beta {\lambda_{j+1} \choose j-i+1}\right)\\
+\beta^{j-i}\alpha^{\lambda_{i}-\lambda_{j}} \sum_{\ell=0}^{\lambda_j-\lambda_{j+1}-1} \alpha^{\ell} \left({\lambda_j-\ell-1 \choose j-i-1}+\beta {\lambda_j-\ell-1 \choose j-i} \right)
\end{multline*}

\end{thm}
    
Observe that similarly to Narayana's formula, for each $m$, the $m \times m$ bottom-right minor of $A^{\alpha,\beta}_{(\lambda_1,\ldots,\lambda_k)}$ equals $A^{\alpha,\beta}_{(\lambda_{k-m+1},\lambda_{k-m+2},\ldots,\lambda_k)}$, whose determinant enumerates the weighted  Catalan paths on the shape $(\lambda_{k-m+1},\lambda_{k-m+2},\ldots,\lambda_k)$.

We prove this formula analogously to the proof of Narayana's path-counting formula by an inductive argument, except that now instead of letting each path have weight 1, we assign the weights to the Catalan paths according to the bijection given in Section \ref{sec_bijection}. In particular, in Narayana's formula, since the weight of each new component is 1, we enumerated the paths by adding up all possible contributions row by row. Similarly, to prove Theorem \ref{main_thm}, we enumerate the possible Catalan paths by, for each successive row of the path, taking the sum of the weight contributions of all the possibilities for that row. The subtlety here is that \emph{the contribution from the segment of the path corresponding to row $j$ is given by the weight of row $j$ in the modified Catalan tableau}. This rule is well-defined since the modified Catalan tableaux are in a weight-preserving bijection with the Catalan paths. In Figure \ref{modified_contributions} we see an example of the weight contributions that arise from each row segment of a Catalan path. Observe that in the modified Catalan tableau, the $\alpha$'s can only lie on the "shelves" of the tableau, i.e. in row $j$ they can only lie in locations $(\lambda_{j+1}+1, \lambda_{j+1}+2,\ldots,\lambda_{j})$, since to create the modified Catalan tableaux, we had all the $\alpha$'s drop the bottom of each column containing them. 

From here on we will no longer be referring to the modified Catalan tableaux, but in the proof that follows, it is assumed that all weight contributions are taken from the rows of the modified Catalan tableaux.

\begin{figure}[h]
\centering
\includegraphics[width=0.9\textwidth]{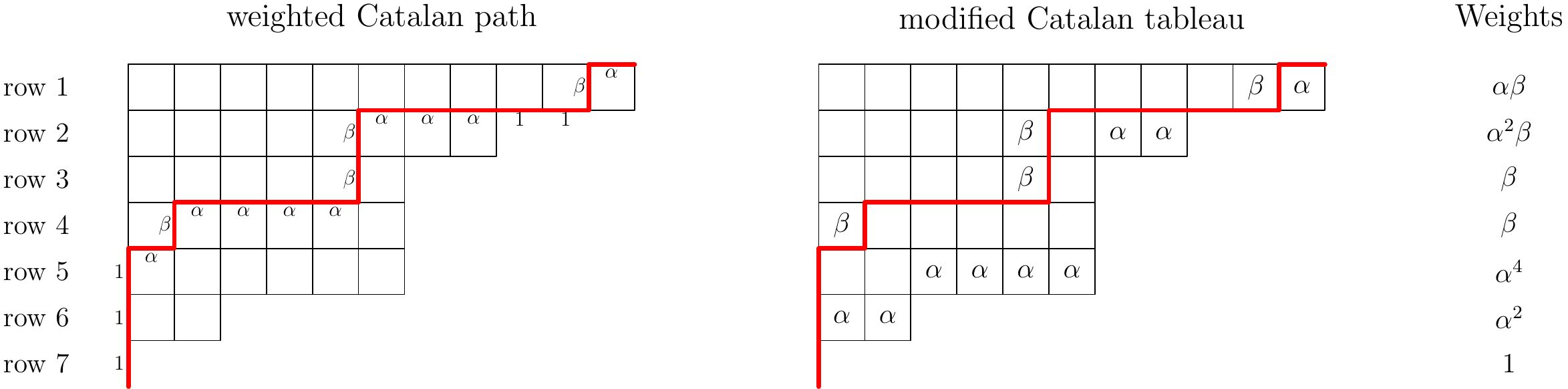}
\caption{The weight contributions from each row of a Catalan path are given by the weight of the corresponding row in the modified Catalan tableau.}
\noindent
\label{modified_contributions}
\end{figure}

In the following, we reproduce the proof of Narayana's path counting formula, on the weighted Catalan paths. 

\begin{proof} First, when $k=1$, a Catalan filling of one row of length $\lambda_1$ can either have some number of $\alpha$'s followed by a $\beta$ when read right to left, or the entire row can be filled with $\alpha$'s. Hence the sum of the weights of Catalan tableau of shape $(\lambda_1)$ is
\begin{equation*}\alpha^{\lambda_1} + \beta \sum_{j=0}^{\lambda_1-1} \alpha^j,\end{equation*}
which equals $\det  A^{\alpha,\beta}_{(\lambda_1)}$.

For $k>1$, we expand the determinant of $A^{\alpha,\beta}_{\lambda}$ by its top row:
\[
\det  A^{\alpha,\beta}_{(\lambda_1,\ldots,\lambda_k)} =  \sum_{i=1}^k (-1)^{i-1} W_i \det A_{(\lambda_{i+1},\dots,\lambda_k)}^{\alpha,\beta}
\]
where $W_j = (A_{\lambda}^{\alpha,\beta})_{1j}$ for $1 \leq j \leq k$.

\textbf{Step 1:} A Catalan path on $(\lambda_1,\ldots,\lambda_k)$ contains a south-step somewhere in the top row including the left-most and right-most edges. We label the position of the possible south-step from left to right by $C_1$ for $0 \leq C_1 \leq \lambda_1$, where $C_1=0$ corresponds to that south-step being on the left-most edge of the tableau and hence has weight 1, and otherwise, that south-step is on the edge to the right of some column $C_1$ and carries weight $\beta$. When $C_1=0$, the 
Catalan path contains no $\beta's$ in the top row, otherwise it contains a $\beta$ in position $C_1$. 
The weight contribution of this segment of the path is thus
\begin{equation*}\begin{cases}
\mbox{(a) } \qquad \qquad \alpha^{\lambda_1-\lambda_2} & C_1=0\\ 
\mbox{(b), (c) } \qquad \beta \alpha^{\lambda_1-\max(\lambda_2, C_1)} & C_1>0.
\end{cases}\end{equation*}
Here and in the steps that follow, the labels (a), (b), (c), (d) refer to the cases in Figure \ref{proof4}.

Therefore, the sum of weights of all possible modified fillings of Row 1, from taking the sum over all choices of $0\leq C_1\leq \lambda_1$, is
\begin{equation*}W_1 = \alpha^{\lambda_1-\lambda_2}(1+\lambda_2\beta)+\beta \sum_{\ell=0}^{\lambda_1-\lambda_2-1} \alpha^{\ell}.\end{equation*}
$W_1 \times \det \ A^{\alpha,\beta}_{(\lambda_2,\ldots,\lambda_k)}$ is the sum of the weights of the combination of all choices for the south-step in the top row, with all possibilities for Catalan paths that start in the northeast corner of the shape $(\lambda_2,\ldots,\lambda_k)$. Here all possible Catalan paths have been accounted for, but we have also included some illegal collections of south-steps that violate Rule 1 at rows 1 and 2.
 
 \textbf{Step 2:} We subtract out illegal combinations of south-steps where $C_1<C_2$ for $C_1$ the location of the south-step in Row 1, and $C_2$ that of Row 2. 
 The ways of selecting this illegal pair of south-steps in rows 1 and 2 is simply a choice of two disjoint columns $C_1$ and $C_2$ with $C_1<C_2$ such that the top row gets a south-step in column $C_1$ and the row below it gets a south-step in  column $C_2$. Figure \ref{proof4} with $j=2$ gives the different cases for the possibilities for $C_1$ and $C_2$ with $C_1<C_2$, with the following weight contributions:
 \begin{equation*}\begin{cases}
\mbox{(a), (b) } \qquad \beta\alpha^{\lambda_1-\max(\lambda_3, C_2)} & C_1=0\\
\mbox{(c), (d) } \qquad \beta^2\alpha^{\lambda_1-\max(\lambda_3, C_2)} & C_1 >0
 \end{cases}\end{equation*}
 
 Taking the sum over all choices of $C_1$ and $C_2$ with $0\leq C_1<C_2\leq \lambda_2$, we obtain
 \begin{equation*} W_2=\beta\alpha^{\lambda_1-\lambda_3}\left({\lambda_3 \choose 1} + \beta {\lambda_3 \choose 2} \right) + \beta\alpha^{\lambda_1-\lambda_2}\sum_{\ell=0}^{\lambda_2-\lambda_3-1}\alpha^{\ell} \left(1+\beta{\lambda_2-\ell-1 \choose 1} \right)\end{equation*}
 $W_2 \times \det \ A^{\alpha,\beta}_{(\lambda_3,\ldots,\lambda_k)}$ is the sum of the weights of all combinations of violations of Rule 1 at rows 1 and 2 with all possible Catalan paths starting from row 3. This is the second term in the expansion of the determinantal formula.

\textbf{Step 3:} We now add back in those combinations of south-steps where Rule 1 is violated not only in rows 1 and 2, but also in rows 2 and 3 (i.e. of the form shown in Figure \ref{illegal} (c)). We enumerate these combinations by selecting $0 \leq C_1<C_2<C_3\leq \lambda_3$ where $C_i$ is the location of the south-step in row $i$. Figure \ref{proof4} for $j=3$ shows the different cases for $C_1, C_2,C_3$ with $C_1<C_2<C_3$, with the following weight contributions:
\begin{equation*}\begin{cases}
 \mbox{(a), (b) } \qquad \beta^2\alpha^{\lambda_1-\max(\lambda_4,C_3)} & C_1=0\\

\mbox{(c), (d) } \qquad \beta^3\alpha^{\lambda_1-\max(\lambda_4,C_3)} & C_1 >0.
 \end{cases}\end{equation*}
 
 \begin{figure}[h]
\centering
\includegraphics[width=\textwidth]{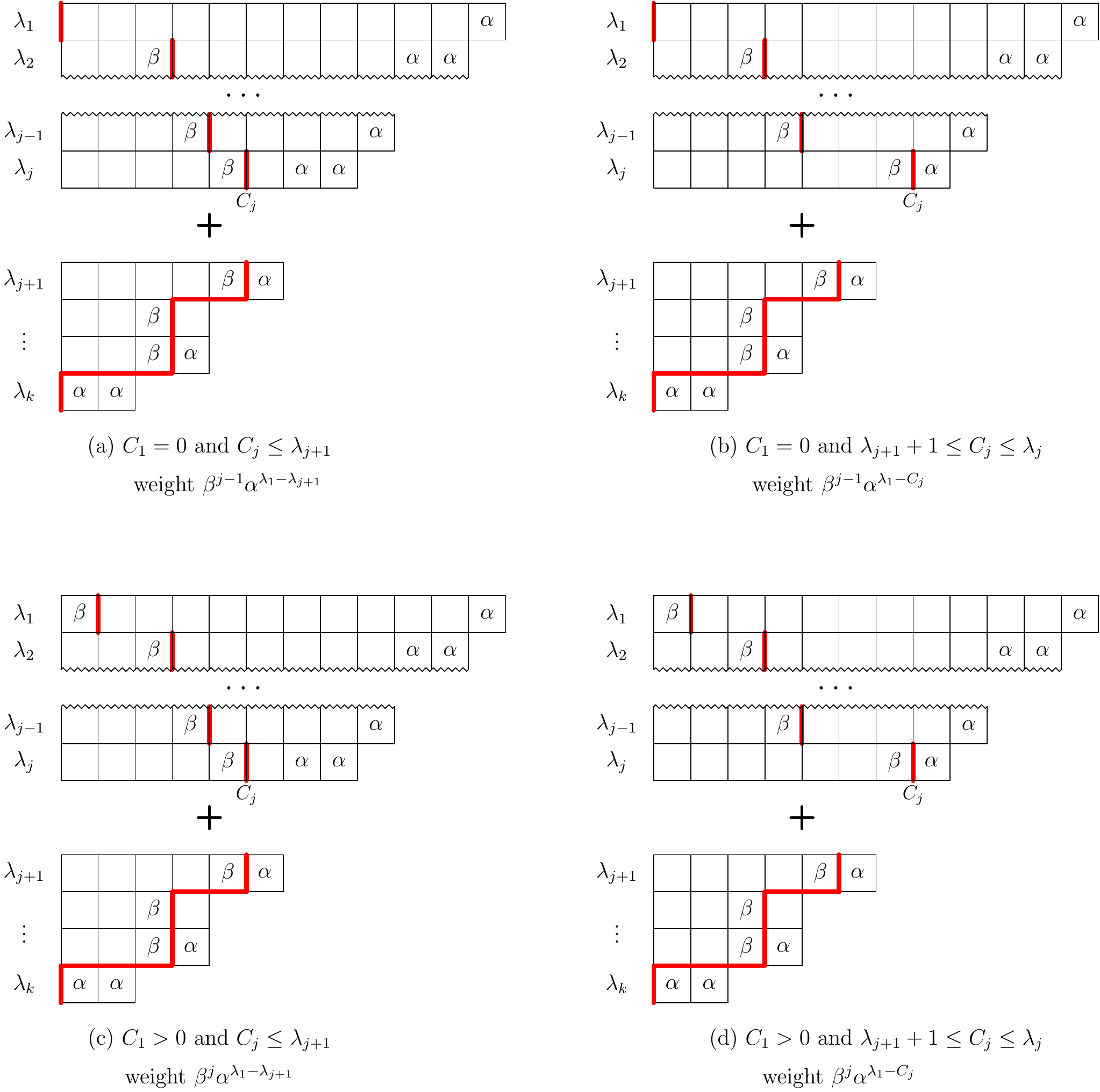}
\caption{The four cases of choosing a set of south-steps at locations $C_1<\cdots<C_j$ in corresponding rows $1,\ldots,j$ and the weights associated to these. }
\noindent
\label{proof4}
\end{figure}
\clearpage

The sum over all such $C_1,C_2,C_3$ is:
\begin{equation*}W_3 = \beta^2\alpha^{\lambda_1-\lambda_4}\left({\lambda_4 \choose 2} + \beta {\lambda_4 \choose 3} \right) + \beta^2\alpha^{\lambda_1-\lambda_3}\sum_{\ell=0}^{\lambda_3-\lambda_4-1} \alpha^{\ell} \left({\lambda_3-\ell-1 \choose 1}+\beta{\lambda_3-\ell-1 \choose 2} \right).\end{equation*}
The possibilities for all paths starting from row 4 are given by $\det A^{\alpha,\beta}_{(\lambda_4,\ldots,\lambda_k)}$, so the product with $W_3$ is the third term of the determinantal expansion.

\textbf{Step 4 through $\boldsymbol{k}$:} We repeat the above for $k$ steps, where at step $j$ we subtract or add the combinations of south-steps that violated Rule 1 in rows 1 and 2, 2 and 3, \ldots, and $(j-1)$ and $j$. We enumerate these combinations by selecting $0 \leq C_1<\cdots<C_j \leq \lambda_3$ where $C_i$ is the location of the south-step in row $i$. Figure \ref{proof4} shows the different cases for $C_1,\ldots,C_j$ with $C_1< \cdots<C_j$, with the following weight contributions:
\begin{equation*}\begin{cases}
\mbox{(a), (b) } \qquad  \beta^{j-1}\alpha^{\lambda_1-\max(\lambda_{j+1},C_j)} & C_1=0\\
 \mbox{(c), (d) } \qquad \beta^j\alpha^{\lambda_1-\max(\lambda_{j+1},C_j)} & C_1 >0 . \end{cases}\end{equation*}

The sum over all such $C_1,\ldots,C_j$ is:
\begin{multline*}
 W_j =\beta^{j-1}\alpha^{\lambda_1-\lambda_{j+1}}\left({\lambda_{j+1} \choose j-1} + \beta {\lambda_{j+1} \choose j} \right) \\
+ \beta^{j-1}\alpha^{\lambda_1-\lambda_j}\sum_{\ell=0}^{\lambda_j-\lambda_{j+1}-1} \alpha^{\ell} \left({\lambda_j-\ell-1 \choose j-2}+\beta{\lambda_j-\ell-1 \choose j-1} \right).
\end{multline*}
This is because, once we have chosen $C_j$, there remain ${C_j-1 \choose j-2}$ choices for $\{C_2,\ldots,C_{j-1}\}$ such that $1\leq C_2<\cdots<C_{j-1}<C_j$ when $C_1=0$, and ${C_j-1 \choose j-1}$ choices for $\{C_1,\ldots,C_{j-1}\}$ such that $1\leq C_1< C_2<\cdots<C_{j-1}<C_j$ when $C_1>0$, and the formula above is obtained by taking the sum over all $C_j$.

 Hence at step $j$, we add the product $(-1)^{j-1}W_j \times \det A^{\alpha,\beta}_{(\lambda_{j+1},\ldots,\lambda_k)}$, the $j$'th term in the determinantal expansion --- thus we have accounted for all the terms in the determinantal formula.
 \end{proof}
 
 \begin{rem} Given that Theorem \ref{main_thm} is a determinantal formula that counts lattice paths, it is tempting to try to prove it using the Karlin-McGregor-Lindstr\"{o}m-Gessel-Viennot Theorem. However, that theorem interprets the determinant of a $k\times k$ matrix in terms of collections of $k$ paths, while in our setting we have a $k \times k$ matrix whose determinant counts \emph{single} paths on shapes of $k$ rows. 
\end{rem}

\begin{cor}
The un-normalized steady state probability that the TASEP with $n$ sites has particles in precisely the locations $1 \leq x_1 < \cdots <x_k\leq n$ is:

\begin{equation*}\mathcal{P}\left[ \{x_1,\ldots,x_k\} \right] = \det A_{\lambda}^{\alpha,\beta} \end{equation*}
 where $A_{\lambda}^{\alpha,\beta}$ is given by
 \begin{multline*}
 A_{ij} = 
 \beta^{j-i}\alpha^{i-(j+1)+x_{j+1}-x_i}\left( {n-k+j+1-x_{j+1} \choose j-i} + \beta {n-k+j+1-x_{j+1} \choose j-i+1} \right)\\
 + \beta^{j-i} \alpha^{i-j+x_j-x_i} \sum_{\ell=0}^{x_{j+1}-x_j-1} \alpha^{\ell} \left( {n-k+j-x_j-\ell-1 \choose j-i-1} + \beta {n-k+j-x_j-\ell-1 \choose j-i} \right)
 \end{multline*}

 \end{cor}
 
\begin{proof} We refer to Theorem \ref{cw_thm} to connect back to the TASEP from the Catalan tableaux. A TASEP state of length $n$ with $k$ particles in locations $\{x_1,\ldots,x_k\}$ corresponds to a word $\tau$ in $\{\bullet,\circ\}^n$ with the $i$th $\bullet$ in location $x_i$. From Definition \ref{condensed_defn}, this state corresponds to Catalan tableaux of shape $\lambda(\tau)=(n-k+1-x_1,n-k+2-x_2,\ldots,n-k+k-x_k)$. Equivalently, $\lambda(\tau)=(\lambda_1,\ldots,\lambda_k)$ where $\lambda_j$ is the number of holes to the right of particle $j$, meaning $\lambda_j=n-k+j-x_j$. Thus Theorem \ref{main_thm} gives the desired formula. 
\end{proof}

\FloatBarrier
\section{Formula for the Condensed Catalan tableaux of size $(k,k+m)$} \label{sec_enum}

Let $N_{m,k}(\alpha,\beta)$ be the weight generating function for the Condensed Catalan tableaux of size $(k,k+m)$. In other words, their associated Young diagrams are contained in a $k\times m$ rectangle.

Let $N'_{m',k'}(\alpha,\beta)$ be the weight generating function for the Condensed Catalan tableaux whose Young diagrams have first row equal to $m'$ and which have precisely $k'$ rows. In other words the Young diagram can be described by the partition $\lambda'=(\lambda'_1,\ldots,\lambda'_{k'})$ where $1 \leq \lambda'_{k'} \leq \cdots \leq \lambda'_1 = m'$. The following gives the relation between $N_{m,k}(\alpha,\beta)$ and $N'_{m',k'}(\alpha,\beta)$:
\begin{equation}\label{sum}
N_{m,k}(\alpha,\beta) = \alpha^{k}\beta^{m} \sum_{m'=0}^m \sum_{k'=0}^k \frac{1}{\alpha^{k'}\beta^{m'}} N'_{m',k'}(\alpha,\beta).
\end{equation}
Here we multiplied by a factor of $\alpha^k \beta^m$ to account for the weight of the lattice path $L(T)$ that is associated with a Condensed Catalan tableau $T$ of size $(k,k+m)$.

Enumerating all the Condensed Catalan tableaux of size $(k,k+m)$ whose Young diagrams have $k$ nonzero rows and first row of length $m$ is equivalent to taking the sum 
\begin{equation*}
N'_{m,k}(\alpha,\beta) = \alpha^k\beta^m \sum_{1 \leq \lambda_k \leq\cdots \leq \lambda_2 \leq m} \det \ A_{\{m,\lambda_2,\lambda_3,\ldots,\lambda_k)}.
\end{equation*}
The above gives rise to the following Theorem.

\begin{thm}\label{mbyk_thm} The weight generating function $N'_{m,k}(\alpha,\beta)$ equals
\begin{equation} \label{aux_formula}
\alpha^k \beta^m \sum_{\ell=0}^k \sum_{j=0}^m \alpha^j \beta^{\ell} \bigg( \binom{m+\ell-2+\delta_{jm}}{m-1}\binom{k+j-2+\delta_{\ell k}}{k-1}-\binom{m+\ell-2+\delta_{jm}}{m}\binom{k+j-2+\delta_{\ell k}}{k} \bigg)
\end{equation}
where $\delta_{rs}$ is the Kronecker $\delta$.
\end{thm}

Summation of \eqref{aux_formula} according to \eqref{sum} yields the following:
\begin{equation} \label{main_formula} N_{m,k}(\alpha,\beta) = \alpha^k\beta^m \sum_{j=0}^{m} \sum_{\ell=0}^{k} \alpha^j \beta^{\ell} \left({k+j-1 \choose j} {m+\ell-1 \choose \ell} -  {k+j-1 \choose j-1} {m+\ell-1 \choose \ell-1}\right) 
\end{equation}

\begin{proof} We prove Formula \eqref{aux_formula} by induction on $m$ and $k$. As seen in Figure \ref{hook_recursion}, a Young diagram with $k$ nonzero rows and with first row of length $m$ can be formed by the addition of a $k-m$ hook with a row of length $m$ and column of length $k$ to the top and left edges of a Condensed Catalan tableau contained in a $k-1 \times m-1$ rectangle.

\begin{figure}[h]
\centering
\includegraphics[width=0.5\textwidth]{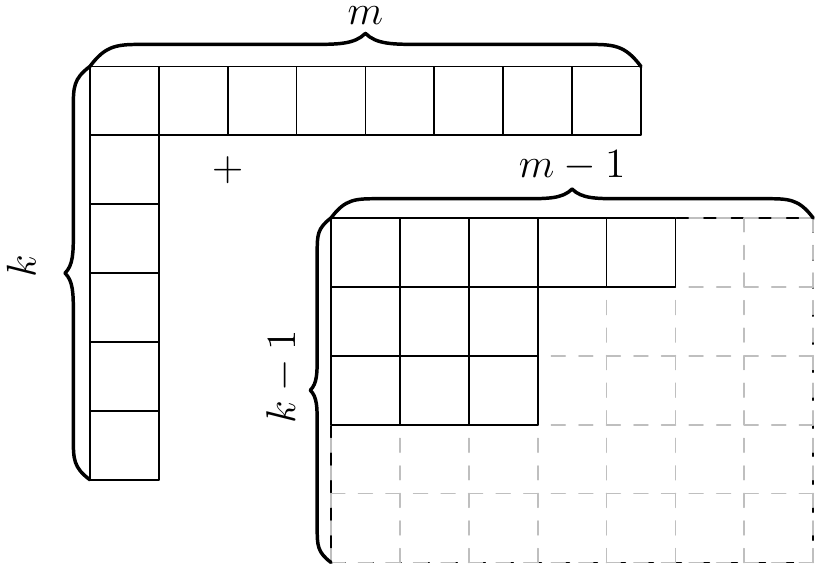}
\caption{Constructing a Catalan tableau with $k$ nonzero rows and first row of length $m$ by adding a $k-m$ hook. }
\noindent
\label{hook_recursion}
\end{figure}

Let $H^{m,k}_{p,q}$ be the sum of the weights of the possible fillings of the $k-m$ hook, when the inside tableau has $p$ rows that are $\alpha$-indexed and $q$ columns that are $\beta$-indexed. If the inside tableau has weight $\alpha^j \beta^{\ell}$, then it must contain $\ell$ $\beta$'s, and so there are $k-1-\ell$ rows that are $\alpha$-indexed since there is always at most one $\beta$ per row. By a similar argument, the inside tableau contains $j$ $\alpha$'s, and hence then there must be $m-1-j$ columns that are $\beta$-indexed, since there is always at most one $\alpha$ per column. Figure \ref{hook} shows the cases that result in the following expression:
\begin{equation}\label{hook_weights}
H^{m,k}_{k-1-\ell,\ m-1-j} = \alpha^{m-j}\sum_{s=0}^{k-\ell-1} \beta^s + \beta^{k-\ell}\sum_{t=0}^{m-j-1} \alpha^t + \sum_{t=1}^{m-j-1}\sum_{s=1}^{k-\ell-1} \alpha^t \beta^s.
\end{equation}
\begin{figure}[h]
\centering
\includegraphics[width=\textwidth]{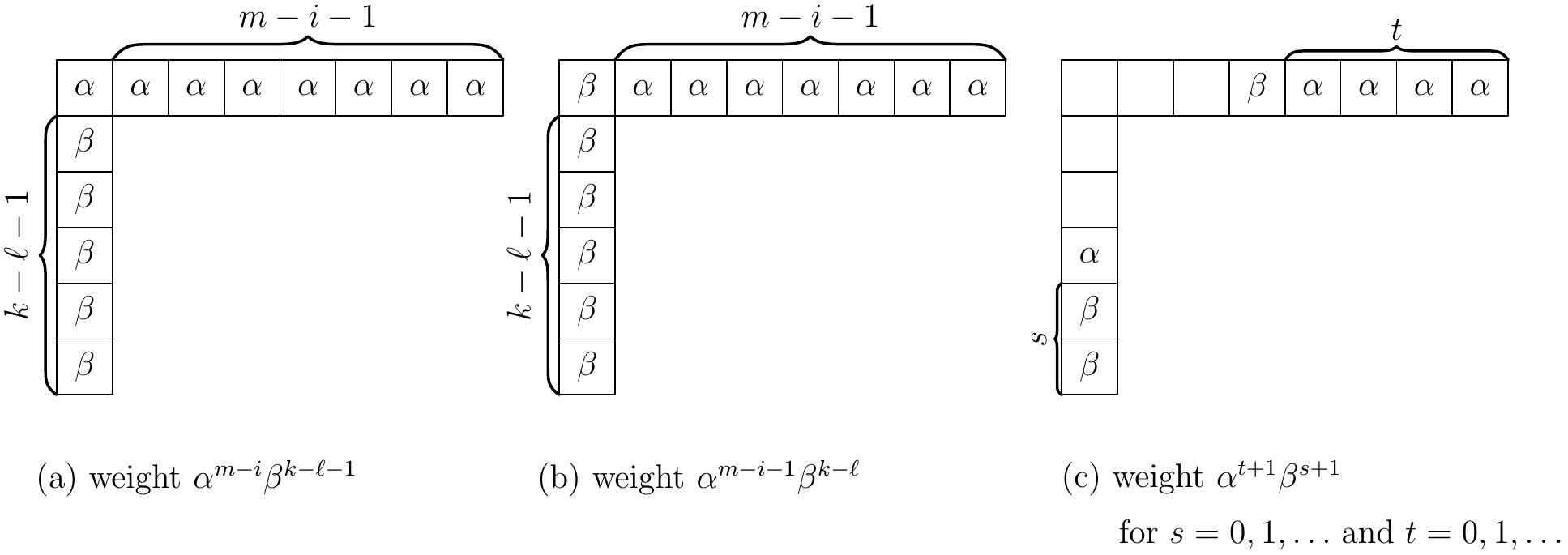}
\caption{The weights for the three cases for Catalan fillings of a $k-m$ hook with $k-1-\ell$ free rows and $m-1-j$ free columns.}
\noindent
\label{hook}
\end{figure}

Recall that if $f(\alpha,\beta)$ is a polynomial in $\alpha$ and $\beta$, then $[\alpha^j\beta^{\ell}]f(\alpha,\beta)$ denotes the coefficient of $\alpha^j\beta^{\ell}$ in $f(\alpha,\beta)$.

Hence for $m,k \geq 2$ we obtain the following recursion:
\begin{equation}\label{N'}
N'_{m,k}(\alpha,\beta) = \alpha^k\beta^m \sum_{j=0}^{m-1} \sum_{\ell=0}^{k-1} H^{m,k}_{k-1-\ell,\ m-1-j}\alpha^j\beta^{\ell}\  \left[\alpha^j \beta^{\ell}\right] \frac{1}{\alpha^{k-1}\beta^{m-1}}N_{m-1,k-1}(\alpha,\beta).
\end{equation}
Note that the coefficient of $\alpha^j \beta^{\ell}$ in $\frac{1}{\alpha^{k-1}\beta^{m-1}}N_{m-1,k-1}(\alpha,\beta)$ gives the number of tableaux contained in an $m-1 \times k-1$ rectangle with $j$ $\alpha$-indexed columns and $\ell$ $\beta$-indexed rows. By the induction hypothesis and from \eqref{main_formula} we know that to be 
\begin{equation*}{k+j-2 \choose j} {m+\ell-2 \choose \ell} -  {k+j-2 \choose j-1} {m+\ell-2 \choose \ell-1}.\end{equation*}
The recursion is now straightforward to verify. On the right hand side of \eqref{N'}, we have
\begin{align*}
& \alpha^k\beta^m \left[ \alpha^m \sum_{l=0}^{k-1} \beta^l  \left( {m+l-1 \choose m-1}{k+m-2 \choose k-1} - {m+l-1 \choose m}{k+m-2 \choose k} \right) \right.\\
&\quad\qquad + \beta^{k} \sum_{j=0}^{m-1} \alpha^j\left( {m+k-2 \choose m-1}{k+j-1 \choose k-1} - {m+k-2 \choose m}{k+j-1 \choose k} \right) \\
&\quad\qquad \left.+ \sum_{l=1}^{k-1} \sum_{j=1}^{m-1} \alpha^j\beta^l  \left( {m+l-2 \choose m-1}{k+j-2 \choose k-1} - {m+l-2 \choose m}{k+j-2 \choose k} \right) \right],
\end{align*}
where we have used that
$\sum_{i=0}^{a} {b+i\choose c} = {b+a+1 \choose c+1} - {b \choose c+1}.$ 

This formula equals \eqref{aux_formula}, which is the left hand side of \eqref{N'} that we desire.

It remains to check the base cases for $N'_{m,k}(\alpha,\beta)$ when $m=1$ or $k=1$. If we plug $m=1$ into \eqref{aux_formula}, we obtain 
\begin{equation*}N'_{1,k}(\alpha,\beta) = \alpha^k \beta \left(\beta^k+\alpha\sum_{\ell=0}^{k-1} \beta^{\ell}\right),\end{equation*}
which is the sum of the weights of Condensed Catalan tableaux of the shape $\lambda=(1,\ldots,1)$ of $k$ rows. Similarly, plugging $k=1$ into \eqref{aux_formula} yields
\begin{equation*}N'_{m,1}(\alpha,\beta) = \alpha \beta^m \left(\alpha^m+\beta\sum_{i=0}^{m-1} \alpha^i\right),\end{equation*}
which is the sum of the weights of Condensed Catalan tableaux of the shape $\lambda=(m)$, and so the proof is complete.
\end{proof}

\begin{defn} Let $Z_n(\alpha,\beta)=\sum_{k=0}^n N_{n-k,k}(\alpha,\beta)$ be the weight generating function for the Condensed Catalan tableaux of size $n$, or equivalently, all Condensed Catalan tableaux that fit in a rectangle of semi-perimeter $n$. 
\end{defn}

\begin{rem}Derrida provides the following formula in \cite{derrida}: 
\begin{equation}\label{Z} 
Z_n(\alpha,\beta) = \alpha^n\beta^n \sum_{p=1}^n \frac{p}{2n-p} {2n-p \choose n} \frac{\alpha^{-p-1}-\beta^{-p-1}}{\alpha^{-1}-\beta^{-1}}.
\end{equation}
This expression normalizes the previously derived stationary probabilities of the TASEP, as we see below in Corollary \ref{cor5.2}. 

Derrida's formula can be derived from \eqref{main_formula} as follows: 
\begin{align}\label{vandermonde}
 [\alpha^{n-t}\beta^{n+t-s}] \sum_{k=0}^n N_{n-k,k}& = \sum_{k=0}^n [\alpha^{n-t}\beta^{n+t-s}] \sum_{j=0}^{n-k} \sum_{\ell=0}^{k} \alpha^{j+k} \beta^{\ell+n-k} \left({k+j-1 \choose k-1} {n-k+\ell-1 \choose n-k-1}\right. \nonumber\\
&\qquad -\left.  {k+j-1 \choose k} {n-k+\ell-1 \choose n-k}\right)\nonumber\\
& = \frac{s}{2n-s}{2n-s \choose n}.
\end{align}

where in the final step the Vandermonde convolution is used.

Since \eqref{vandermonde} is independent of $t$, we obtain \eqref{Z} by summing over $k$.
\end{rem}

\begin{cor}\label{cor5.2}
The stationary probability of a TASEP of length $n$ and containing exactly $k$ particles is $N_{n-k,k}(\alpha,\beta)$ of \eqref{main_formula}, normalized by $Z_{n}(\alpha,\beta)$ from \eqref{Z}.
\end{cor}



\begin{table}
\begin{center} 
{\small 
\begin{tabular}{ | c | c || p{8cm} | p{6cm} | }
\hline
  $n$&$k$&$q^{-n}N_{n-k,k}(q,q)$&$q^{-n}N_{n-k,k}(q,1)=q^{-n}N_{k,n-k}(1,q)$\\
  \hline   \hline                    
  6&$1$ & $2q^5 + 3q^4 + 4q^3 + 5q^2 + 6q + 1$ & $q^5 + 2q^4 + 3q^3 + 4q^2 + 5q + 6$ \\
  6&$2$ & $20q^5 + 30q^4 + 28q^3 + 20q^2 + 6q + 1$ & $15q^4 + 24q^3 + 27q^2 + 24q + 15$ \\
  6&$3$ & $40q^5 + 60q^4 + 48q^3 + 20q^2 + 6q + 1$ & $50q^3 + 60q^2 + 45q + 20$ \\
  6&$4$ & $20q^5 + 30q^4 + 28q^3 + 20q^2 + 6q + 1$ & $50q^2 + 40q + 15$ \\
  6&$5$ & $2q^5 + 3q^4 + 4q^3 + 5q^2 + 6q + 1$ & $15q + 6$ \\
  \hline
  7&$1$ & $2q^6 + 3q^5 + 4q^4 + 5q^3 + 6q^2 + 7q + 1$ & $q^6 + 2q^5 + 3q^4 + 4q^3 + 5q^2 + 6q + 7$ \\
  7&$2$ & $30q^6 + 45q^5 +46q^4 + 40q^3 + 27q^2 + 7q + 1$ & $21q^5 + 35q^4 + 42q^3
+ 42q^2 + 35q + 21$ \\
  7&$3$ & $100q^6 + 150q^5 + 130q^4 + 75q^3+ 27q^2 + 7q + 1$ & $105q^4 + 140q^3 + 126q^2 + 84q + 35$ \\
  7&$4$ & $100q^6 + 150q^5 + 130q^4 + 75q^3 + 27q^2 + 7q+ 1$ & $175q^3 +
175q^2 + 105q + 35$ \\
  7&$5$ & $30q^6 + 45q^5 + 46q^4 + 40q^3 + 27q^2 + 7q + 1$ & $105q^2 + 70q + 21$ \\
  7&$6$ & $2q^6 + 3q^5+ 4q^4 + 5q^3 + 6q^2 + 7q + 1$ & $21q + 7$ \\
  \hline
  8&$1$ & $2q^7 + 3q^6 + 4q^5 + 5q^4 + 6q^3 + 7q^2 + 8q + 1$ & $q^7 + 2q^6 + 3q^5 + 4q^4 + 5q^3 + 6q^2 + 7q + 8$ \\
  8&$2$ & $42q^7 +63q^6 + 68q^5 + 65q^4 + 54q^3 + 35q^2 + 8q + 1$ & $28q^6 + 48q^5
+ 60q^4 + 64q^3 + 60q^2 + 48q + 28$ \\
  8&$3$ & $210q^7 + 315q^6+ 292q^5 + 205q^4 + 110q^3 + 35q^2 + 8q + 1$ & $196q^5 + 280q^4 + 280q^3 +
224q^2 + 140q + 56$ \\
  8&$4$ & $350q^7 + 525q^6 +460q^5 + 275q^4 + 110q^3 + 35q^2 + 8q + 1$ & $490q^4 + 560q^3 + 420q^2 + 224q + 70$ \\
  8&$5$ & $210q^7 + 315q^6 +292q^5 + 205q^4 + 110q^3 + 35q^2 + 8q + 1$ & $490q^3
+ 420q^2 + 210q + 56$ \\
  8&$6$ & $42q^7 + 63q^6 + 68q^5+ 65q^4 + 54q^3 + 35q^2 + 8q + 1$ & $196q^2 + 112q + 28$ \\
  8&$7$ & $2q^7 + 3q^6 + 4q^5 + 5q^4 +6q^3 + 7q^2 + 8q + 1$ & $28q + 8$ \\
  \hline  
\end{tabular}} 
\caption{Specialization of the $\alpha/\beta$-Narayana numbers $N_{n-k,k}(\alpha,\beta)$.} \label{narayana_polys}
\end{center}
\end{table}

\begin{rem}[\cite{duchi, advances2005}] Using standard binomial identities it follows from Theorem \ref{mbyk_thm} that $N_{n-k,k}$ is an $\alpha\ /\ \beta$ generalization of the Narayana numbers $\mathcal{N}(n+1,k+1)$. That is, $N_{n-k,k}(\alpha=1,\beta=1) = \frac{1}{n+1} {n+1 \choose k}{n+1 \choose k+1}$. Consequently the total number of Catalan tableaux of size $n$ is $Z_n(1,1)=\frac{1}{n+2}{2n+2 \choose n+1}$, the Catalan number $C_{n+1}$.
\end{rem}

\begin{rem} The above remark can be extended to provide a $q$-refinement of the Narayana numbers by setting $\alpha=\beta=q$ or $\alpha=1$ and $\beta=q$ in \eqref{main_formula}. Table 1 shows some of the resulting $q$-polynomials. We include these particular specializations because the polynomials $N_{n-k,k}(\alpha,\beta)$ contain on the order of $n(n-k)$ terms, and they quickly get very long. 
\end{rem}


\end{document}